\documentclass[12pt]{article} 

\usepackage[utf8]{inputenc} 
\usepackage[T1]{fontenc}    
\usepackage{hyperref}       
\usepackage{url}            
\usepackage{booktabs}       
\usepackage{amsfonts}       
\usepackage{nicefrac}       
\usepackage{microtype}      
\usepackage{mathrsfs}
\usepackage{indentfirst}
\usepackage{multirow}
\usepackage{subfigure}
\usepackage{amsmath}
\usepackage{amssymb}
\usepackage{natbib}
\usepackage{graphicx}
\usepackage{url}
\usepackage{authblk}
\usepackage{geometry}
\usepackage{multirow}
\usepackage{subfigure}
\usepackage[linesnumbered,ruled,vlined]{algorithm2e}
\newenvironment{proof}{\paragraph{Proof:}}{\hfill$\square$}
\newtheorem{theorem}{Theorem}[section]

\newtheorem{definition}[theorem]{Definition}
\newtheorem{remark}[theorem]{Remark}

\newtheorem{corollary}[theorem]{Corollary}
\usepackage{xcolor}
\usepackage{titlesec}
\usepackage{footmisc}
\title{On The Convergence of First Order Methods for Quasar-Convex Optimization}

\author{Jikai Jin \\
School of Mathematical Sciences, Peking University \\
\texttt{1900010640@pku.edu.cn}}

\begin{document}

\maketitle

\begin{abstract}%
In recent years, the success of deep learning has inspired many researchers to study the optimization of general smooth non-convex functions. However, recent works have established pessimistic worst-case complexities for this class functions, which is in stark contrast with their superior performance in real-world applications (e.g. training deep neural networks). On the other hand, it is found that many popular non-convex optimization problems enjoy certain structured properties which bear some similarities to convexity. In this paper, we study the class of \textit{quasar-convex functions} to close the gap between theory and practice. We study the convergence of first order methods in a variety of different settings and under different optimality criterions. We prove complexity upper bounds that are similar to standard results established for convex functions and much better that state-of-the-art convergence rates of non-convex functions. Overall, this paper suggests that \textit{quasar-convexity} allows efficient optimization procedures, and we are looking forward to seeing more problems that demonstrate similar properties in practice.
\end{abstract}


\section{Introduction}
In this paper we consider the problem of minimizing a given objective function $f: \mathbb{R}^n \to \mathbb{R}$.

The study of this optimization problem has a long history. Early works mainly focus on the special case when $f$ is convex, and we have access to the exact gradient at each point \citep{nesterov2013introductory}. However, the situation begins to change with the advent of the big data era, and, in particular, the rise of machine learning. In many machine learning applications ( e.g. deep neural networks \citep{goodfellow2016deep} ), the objective function is highly complicated and non-convex. Therefore, the classical theory of convex optimization can no longer produce meaningful implications in many real-world scenarios.

Motivated by the empirical success of optimization algorithms, in recent years there has been a flurry of research works that study algorithms in non-convex optimization \citep{ghadimi2013stochastic,kingma2014adam,ghadimi2016accelerated,carmon2018accelerated}. Specifically ,these works study efficient algorithms for finding an approximate stationary points for general smooth non-convex function. A standard result is that the simple Stochastic Gradient Descent(SGD) algorithm can find an $\epsilon$-stationary point with a complexity of $\mathcal{O}\left( \epsilon^{-4} \right)$. However, it has been established recently that this complexity is already optimal among \textit{first order methods} (i.e. methods that only use first-order information of the objective function) \citep{arjevani2019lower}. This gives a convergence rate which is considerably slower than the actual convergence rate we observe in practice, thereby suggesting that there is still a gap between theory and practice.

On the other hand, the study of specific optimization problems suggests that sometimes the objective function exhibits certain desirable properties. For instance it has been proved that there is no spurious local minima in a variety of low-rank matrix problems \citep{ge2017no}, policy optimization in reinforcement learning satisfies some Polyak-Łojasiewicz-type conditions \citep{mei2020global,cen2020fast}, the landscape of neural network exhibits some convex-like properties \citep{li2017convergence}, etc. These observations inspire us to consider the possibility of more efficient optimization when imposing structural assumptions to the objective function.

In this paper we study the optimization using first order methods under \textit{quasar-convexity}, which is a generalization of the notion of \textit{convexity}. While several prior works \citep{hinder2020near,gower2020sgd,bu2020note} have provided theoretical treatments of this class of functions in some specific settings, we extend the analysis of quasar-convex optimization to include a variety of setups that are of practical interest. We also provide sharper results compared with \citep{gower2020sgd} in some special cases. The main results of this paper are summarized below .
\\
\\

\begin{tabular}{ccc}
\hline
                                                      & \multicolumn{2}{c}{Smooth Quasar-Convex Function}                                                                                                                                                                                                                     \\ 
                                                      & Deterministic                                                                                                         & Stochastic                                                                                                                                    \\ \hline
\multirow{2}{*}{Finding Approximate Minima}           & $\widetilde{\mathcal{O}}\left( \sqrt{\frac{LR^2}{\gamma\epsilon}} \right)$                                            & $\mathcal{O}\left( \frac{R^2\sigma^2}{\gamma^2\epsilon^2} + \frac{R^2L}{\gamma\epsilon} \right)$                                              \\  
                                                      & (~\cite[Theorem 2]{hinder2020near})                                                                                   & (Corollary \ref{convex_complexity})                                                                                                           \\ \hline
\multirow{2}{*}{Making Gradient Small}                & $\mathcal{O}\left( LR^{\frac{2}{3}}\gamma^{-\frac{1}{3}}\epsilon^{-\frac{2}{3}} \right)$                             & $\mathcal{O}\left( \sigma^2 \left( \frac{LR}{\gamma\epsilon^4} \right)^{\frac{2}{3}} \right)$                                                 \\ 
                                                      & (Theorem \ref{qc_det_grad})                                                                                           & (Theorem \ref{qc_sto_grad})                                                                                                                   \\ \hline
                                                      & \multicolumn{2}{c}{Smooth Strongly-Quasar-Convex Function}                                                                                                                                                                                                            \\ 
                                                      & Deterministic                                                                                                         & Stochastic                                                                                                                                    \\ \hline
\multirow{2}{*}{Finding Approximate Minima}           & $\mathcal{O}\left( \sqrt{\frac{\kappa}{\gamma^2}} \log \left( \frac{\kappa \Delta}{\epsilon} \right) \right)$         & $\widetilde{\mathcal{O}}\left( \frac{\sigma^2}{\gamma^2\mu\epsilon} \right)$ \\ 
                                                      & (~\cite[Theorem 1]{hinder2020near})                                                                                   & (Theorem \ref{sqc_sto_complexity})                                                                                                          \\ \hline
\multirow{2}{*}{Making Gradient Small} & $\widetilde{\mathcal{O}}\left( \sqrt{\frac{\kappa}{\gamma^2}} \log \left( \frac{\sqrt{\kappa} LR }{\epsilon} \right) \right)$ & $\widetilde{\mathcal{O}}\left( \sqrt{\frac{L}{\mu}} \frac{\sigma^2}{\gamma\epsilon^2} \right)$                                                \\ 
                                                      & (Corollary \ref{simple})                                                                                              & (Theorem \ref{sqc_sto_grad})                                                                                                                  \\ 
\end{tabular}

\section{Preliminaries}

We first introduce some definitions that will be useful in this paper.

\begin{definition}
\label{Lsmooth}
A differentiable function $f : \mathbb{R}^n \to \mathbb{R}$ is said to be $L$-smooth if its gradient is $L$-Lipschitz, i.e. $\| \nabla f(x) - \nabla f(y) \| \leq L \| x-y \|$ for all $x,y \in \mathbb{R}^n$.
\end{definition}

\begin{definition}
(Quasar-convexity) Suppose that the function $f : \mathbb{R}^n \to \mathbb{R}$ is differentiable and has a global minimizer $x^*$, then we say that $f$ is $\gamma$-quasar-convex w.r.t. $x^*$ if $0 \leq \gamma \leq 1$ and the following holds for all $x \in \mathbb{R}^n$:
\begin{equation}
\label{qc}
    f(x^*) \geq f(x) + \frac{1}{\gamma} \nabla f(x)^T (x^*-x)
\end{equation}
Further we say that $f$ is $(\gamma ,\mu)$-strongly-quasar-convex if additionally $\mu \geq 0$ and the following holds for all $x \in \mathbb{R}^n$:
\begin{equation}
    f(x^*) \geq f(x) + \frac{1}{\gamma} \nabla f(x)^T (x^*-x) + \frac{\mu}{2} \left\| x-x^* \right\|^2
\end{equation}
\end{definition}

Throughout this paper we describe the performance of optimization algorithms via providing their \textit{oracle complexities}. Roughly speaking, the \textit{oracle complexity} of an algorithm is the minimum time it needs to query a certain oracle in order to meet some optimality criterion.  

In this paper we consider two classes of oracles:

\textbf{Deterministic Oracle} $\mathcal{D}(f)$. The algorithm sends a point $x$ to the oracle and the oracle responds with a pair $\left(f(x),\nabla f(x) \right)$.

\textbf{Stochastic Oracle} $\mathcal{S}(f,\sigma )$. The algorithm sends a point $x$ to the oracle and the oracle responds with random vector $g(x)$ such that $\mathbb{E}g(x) = f(x)$ and $\mathbb{E} \|g(x)-\nabla f(x)\|^2 \leq \sigma^2$.

For briefness we will refer to the optimization problem equipped with these two oracles as \textit{deterministic setting} and \textit{stochastic setting}, respectively. 

We will also consider two types of optimality criterion. Specifically, we consider finding an $\epsilon$-optimal point ( i.e. a point $\widetilde{x}$ with $f(\widetilde{x})-\inf_{x}f(x) \leq \epsilon$ ) and finding an $\epsilon$-stationary point ( i.e. a point with $\|\nabla f(x) \| \leq \epsilon$ ).

Formal definitions of related concepts are given in Appendix \ref{def} due to space limits.

\subsection{Notations}

Throughout this paper we let $x_0$ be the starting point of all the algorithms we consider, $x^*$ be the global minima of function $f$, and let $R, \Delta$ be upper bounds of the quantities $\|x_0-x^*\|$ and $f(x_0)-f(x^*)$ respectively. We use $\mathcal{O}$ to hide numerical constants and $\widetilde{\mathcal{O}}$ to hide $\log$ terms. 

\section{Convergence Results for Smooth Quasar-Convex Functions}
\label{sec_qc}

\subsection{Deterministic Setting}

In \citet{hinder2020near} the authors propose a near optimal method for finding $\epsilon$-optimal points in this setting. We recall their result below.

\begin{theorem}
\label{Hinder1}
( ~\cite[Theorem 2]{hinder2020near} ) There exists an algorithm $\mathcal{A}_c$ such that its complexity of finding an $\epsilon$-optimal point is $\mathcal{O} \left( \sqrt{\frac{LR^2}{\gamma \epsilon}} \log \left( \sqrt{\frac{LR^2}{\gamma \epsilon}} \right) \right)$. 
\end{theorem}

Regarding the goal of finding an $\epsilon$-stationary point, a direct application of the inequality $\|\nabla f(x)\|^2 \leq L \left( f(x)-f(x^*) \right)$ would give a $\widetilde{\mathcal{O}}\left( \epsilon^{-1} \right)$ complexity upper bound. However, we can improve this bound by using the \textit{GD after AGD} trick proposed by \citet{nesterov2012make}. The result is summarized in the following theorem.

\begin{theorem}
\label{qc_det_grad}
There exists an algorithm such that for $L$-smooth and $\gamma$-quasi convex functions its complexity of finding $\epsilon$-stationary point is $\widetilde{\mathcal{O}} \left( LR^{\frac{2}{3}}\gamma^{-\frac{1}{3}}\epsilon^{-\frac{2}{3}}\right)$.
\end{theorem}

\begin{proof}
The idea is to use Nesterov's \textit{'GD after AGD'} trick \citep{allen2018make}, where we replace Nesterov's AGD with the algorithm $\mathcal{A}_c$.

Specifically, for a fixed $\epsilon_1>0$, we first run $\mathcal{A}_c$ for $\widetilde{\mathcal{O}} \left( \sqrt{\dfrac{LR^2}{\gamma \epsilon_1}} \right)$ iterations and arrive at a point $\widetilde{x_0}$ such that $f\left( \widetilde{x_0} \right) - f(x^*) \leq \epsilon_1$. Then, starting from $\widetilde{x_0}$ we run gradient descent for $\mathcal{O} \left(L\epsilon_1 \epsilon^{-2}\right)$ iterations. It is well known that we can then find a point $x$ such that $\|\nabla f(x)\| \leq \epsilon$.

The complexity of the above procedure is then
$$\widetilde{\mathcal{O}} \left( \sqrt{\dfrac{LR^2}{\gamma \epsilon_1}} + L\epsilon_1 \epsilon^{-2} \right)$$
The result follows by choosing $\epsilon_1 = \left( \gamma^{-1} R^2\epsilon^4 \right)^{\frac{1}{3}}$.
\end{proof}

\subsection{Stochastic Setting}

 We first recall the vanilla Stochastic Gradient Descent(SGD) algorithm:
 
 \begin{algorithm}[!htbp]
\label{SGD}
\SetKwInOut{KIN}{Input}
\caption{$\mathtt{SGD} \left( f,x_0,\left\{\alpha_k \right\}_{k\geq 1},T \right)$ }
\KIN{Objective function $f$, initial point $x_0$, parameters $\gamma,L,\sigma$, total iterations $T$}
\For{$t \gets 1$ {to} $T$}{
    $x_t \gets x_{t-1} - \alpha_t \nabla f(x_{t-1},\xi_{t-1})$\;
}
\SetKwInOut{KOUT}{Output}
\KOUT{$\widetilde{x} \in \{ x_1,x_2,\cdots ,x_T \}$ uniformly at random}
\end{algorithm}

 In the following theorem we establish the convergence rate of SGD for smooth quasar-convex functions.

\begin{theorem}
\label{convex_subopt}
Suppose that $f$ is an $L$-smooth, $\gamma$-quasar-convex function, and we run SGD for $T$ iterations with some fixed step size $\alpha_t = \alpha$, where $R = \|x_0-x^*\|$. Then, we have
\begin{equation}
    \frac{1}{T} \mathbb{E} \sum_{t=1}^{T} \left(f(x_t)-f(x^*) \right)  \leq 4 \left( \frac{R\sigma}{\gamma\sqrt{T}} + \frac{1}{\gamma}\frac{R^2L}{T} \right)
\end{equation}
\end{theorem}

\begin{proof}
First note that 
\begin{equation}\begin{aligned}
\label{dist}
    \|x_{t}-x^*\|^2 &= \|x_{t+1}-x^*\|^2 + 2 \left\langle x_t - x_{t+1}, x_{t+1} - x^* \right\rangle + \|x_{t+1}-x_t\|^2 \\
    &= \|x_{t+1}-x^*\|^2 + \|x_{t+1}-x_t\|^2 + 2\alpha \left\langle \nabla f(x_t,\xi_t), x_{t+1} - x^* \right\rangle
\end{aligned}\end{equation}
Denote $\Delta_t = \mathbb{E} \left[ f(x_t)-f^* \right]$, then we have
\begin{equation}
\label{Delta}
    \begin{aligned}
    &\Delta_{t+1} \leq \mathbb{E} \left[ f(x_t) + \left\langle \nabla f(x_t), x_{t+1}-x_t \right\rangle + \frac{L}{2} \|x_{t+1}-x_t \|^2 - f(x^*)\right] \\
    &\leq \mathbb{E} \left[ \left\langle \nabla f(x_t), x_{t+1}-x_t \right\rangle + \frac{L}{2} \|x_{t+1}-x_t \|^2 + \frac{1}{\gamma} \left\langle\nabla f(x_t), x_t-x^* \right\rangle  \right] \\
    &= \mathbb{E} \left[ \left\langle \nabla f(x_t), x_{t+1} + \left( \frac{1}{\gamma}-1 \right) x_t - \frac{1}{\gamma} x^* \right\rangle + \frac{L}{2} \|x_{t+1}-x_t \|^2 \right] \\
    &= \mathbb{E} \left[ \frac{1}{\gamma} \left\langle \nabla f(x_t), x_{t+1}-x^* \right\rangle + \left( \frac{1}{\gamma}-1 \right) \left\langle \nabla f(x_t), x_t-x_{t+1} \right\rangle + \frac{L}{2} \|x_{t+1}-x_t \|^2 \right]
    \end{aligned}
\end{equation}

Now we handle the first and second term in the above expression respectively. First by \eqref{dist}, for some fixed $\lambda > 0$ we have 
\begin{equation}
\label{1st_term}
\begin{aligned}
&\quad \mathbb{E} \left[ \left\langle \nabla f(x_t), x_{t+1}-x^* \right\rangle \right] \\
&= \mathbb{E} \left[ \left\langle \nabla f(x_t) - \nabla f(x_t,\xi_t), x_{t+1}-x^* \right\rangle \right] + \mathbb{E}\left[ \left\langle \nabla f(x_t,\xi_t), x_{t+1}-x^* \right\rangle \right] \\
&= \mathbb{E} \left[ \left\langle \nabla f(x_t) - \nabla f(x_t,\xi_t), x_{t+1}-x_t \right\rangle \right] + \mathbb{E}\left[ \left\langle \nabla f(x_t,\xi_t), x_{t+1}-x^* \right\rangle \right] \\
&\leq \frac{1}{2\lambda} \mathbb{E} \| \nabla f(x_t) - \nabla f(x_t,\xi_t) \|^2 + \mathbb{E} \left[ \frac{\lambda}{2} \|x_{t+1}-x_t\|^2 + \frac{1}{2\alpha} \left( \|x_t-x^*\|^2-\|x_{t+1}-x^*\|^2-\|x_{t+1}-x_t\|^2 \right) \right] \\
&\leq \frac{\sigma^2}{2\lambda} + \mathbb{E} \left[ \frac{1}{2\alpha} \|x_t-x^*\|^2 - \frac{1}{2\alpha} \|x_{t+1}-x^*\|^2 - \left( \frac{1}{2\alpha} - \frac{\lambda}{2} \right) \|x_{t+1}-x_t\|^2 \right]
\end{aligned}
\end{equation}
Next, as long as $\alpha < \frac{1}{2L}$, by the $L$-smoothness of $f$ we have
\begin{equation}
    \begin{aligned}
    &\quad\Delta_{t+1}-\Delta_t \leq -\alpha \mathbb{E}\|\nabla f(x_t)\|^2 + \frac{L}{2}\alpha^2 \left( \mathbb{E} \|\nabla f(x_t)\|^2 +\sigma^2 \right) \leq - \frac{\alpha}{2} \mathbb{E}\|\nabla f(x_t)\|^2 + \frac{L}{2}\alpha^2\sigma^2
    \end{aligned}
\end{equation}
Therefore
\begin{equation}
\label{2nd_term}
    \begin{aligned}
    &\quad \mathbb{E} \left\langle \nabla f(x_t),x_t-x_{t+1} \right\rangle 
    = \alpha \mathbb{E} \|\nabla f(x_t)\|^2 \\
    &\leq 2 \left( \Delta_t - \Delta_{t+1} \right) + L\alpha^2\sigma^2
    \end{aligned}
\end{equation}
Now, by plugging \eqref{1st_term} and \eqref{2nd_term} into \eqref{Delta} we have
\begin{equation}
\begin{aligned}
 \Delta_{t+1} &\leq \frac{\sigma^2}{2\gamma \lambda} + \frac{1}{2\gamma\alpha} \mathbb{E} \left[ \|x_t-x^*\|^2-\|x_{t+1}-x^*\|^2 \right] - \frac{1}{\gamma} \left( \frac{1}{2\alpha} - \frac{\lambda}{2} - \frac{L\gamma}{2} \right) \mathbb{E}\|x_{t+1}-x_t\|^2 \\
    &+ 2\left( \frac{1}{\gamma} -1 \right) (\Delta_t - \Delta_{t+1}) + \left( \frac{1}{\gamma} -1 \right) L \alpha^2\sigma^2
\end{aligned}
\end{equation}
Choosing $\lambda = \frac{1}{\alpha} - L\gamma$ in the RHS and rearranging, we obtain 
\begin{equation}
    \frac{1}{2\gamma\alpha} \mathbb{E} \left[ \|x_t-x^*\|^2-\|x_{t+1}-x^*\|^2 \right] + \frac{\sigma^2\alpha}{2\gamma(1-\alpha L\gamma)} + \left(\frac{1}{\gamma}-1 \right) L\alpha^2\sigma^2 \geq \left( \frac{2}{\gamma}-1 \right) \Delta_{t+1} - \left( \frac{2}{\gamma}-2 \right) \Delta_t 
\end{equation}
Perform a telescope sum  for $t=0,1,\cdots ,T-1$ gives
\begin{equation}
\label{telescope}
    \left( \frac{2}{\gamma}-1 \right) \Delta_{T} + \sum_{t=1}^{T-1} \Delta_t \leq \frac{R^2}{2\gamma\alpha} + \frac{\sigma^2\alpha}{2\gamma(1-\alpha L\gamma)}T +  \left(\frac{1}{\gamma}-1 \right) L\alpha^2\sigma^2 T + \left( \frac{2}{\gamma}-2 \right) \Delta_0
\end{equation}
Again by $L$-smoothness we can see that $\Delta_0 \leq \frac{L}{2} R^2$, thus
\begin{equation}
\label{sum_Delta}
    \sum_{t=1}^{T} \Delta_t \leq \frac{R^2}{2\gamma\alpha} + \frac{\sigma^2\alpha}{2\gamma(1-\alpha L\gamma)}T +  \left(\frac{1}{\gamma}-1 \right) L\alpha^2\sigma^2 T + \left( \frac{1}{\gamma}-1 \right) LR^2
\end{equation}
Choose $\alpha = \frac{R}{2\sigma\sqrt{T}}$, which is smaller than $\frac{1}{2L} $ when $T> \frac{R^2L^2}{\sigma^2}$, then
\begin{equation}
    \frac{1}{T} \sum_{t=1}^{T} \Delta_t \leq \frac{2R\sigma}{\gamma\sqrt{T}} + 2 \left(\frac{1}{\gamma}-1 \right)\frac{R^2L}{T} < \frac{4R\sigma}{\gamma\sqrt{T}}
\end{equation}
Otherwise if $T \leq \frac{R^2L^2}{\sigma^2}$, it follows from \eqref{sum_Delta} that
\begin{equation}
    \frac{1}{T} \sum_{t=1}^{T} \Delta_t \leq \frac{R^2}{2\gamma\alpha T} + \frac{\sigma^2\alpha}{2\gamma(1-\alpha L\gamma)} +  \left(\frac{1}{\gamma}-1 \right) L\alpha^2\sigma^2  + \left( \frac{1}{\gamma}-1 \right) \frac{LR^2}{T}
\end{equation}
We choose $\alpha = \frac{1}{2L}$ in the above inequality, so that by some calculation we have
\begin{equation}
    \frac{1}{T} \sum_{t=1}^{T} \Delta_t \leq 4 \left( \frac{R\sigma}{\gamma\sqrt{T}} + \frac{1}{\gamma}\frac{R^2L}{T} \right)
\end{equation}
\end{proof}

\begin{remark}
The only existing convergence guarantee in our setting that we are aware of is established in \citet{gower2020sgd}. Specifically, they show that in $T$ iterations SGD ensures that 
\begin{equation}
    \frac{1}{T} \sum_{t=0}^{T-1} \mathbb{E}\left[ f(x_k)-f(x^*) \right] \leq \frac{R^2+2\beta^2\sigma^2}{\beta\sqrt{T}}
\end{equation}
for some $\beta \in \left( 0,\frac{\gamma}{L} \right)$. To achieve this convergence rate the step size is chosen as $\alpha = \frac{\beta}{\sqrt{T}}$. We observe that if the step size is chosen more carefully, then their approach can yield a convergence rate which is the same as ours in dominating terms. A detailed discussion is deferred to Appendix \ref{comparison}.
\end{remark}

\begin{corollary}
\label{convex_complexity}
For the class of $L$-smooth, $\gamma$-quasar-convex functions the complexity os SGD for finding $\epsilon$-optimal point is $\mathcal{O} \left( \frac{R^2\sigma^2}{\gamma^2\epsilon^2} + \frac{R^2L}{\gamma\epsilon} \right)$ 
\end{corollary}

Equipped with the above results, we can now establish a complexity upper bound for making gradient small. The idea is to use the \textit{SGD after SGD} approach proposed by \citet{allen2018make}, which is a stochastic counterpart of \textit{GD after SGD}. We first run SGD to find an $\epsilon_1$-optimal point, then run SGD starting from this point to find a point with small gradient.

\begin{theorem}
\label{qc_sto_grad}
There exists an algorithm such that for $L$-smooth and $\gamma$-quasi convex functions its complexity of finding $\epsilon$-stationary point is  $\mathcal{O} \left( \sigma^2 \left( \frac{LR}{\gamma\epsilon^4} \right)^{\frac{2}{3}} \right)$  (here we omit the lower order terms for convenience).
\end{theorem}

\begin{proof}
 Fixed $\epsilon_1>0$, we first run SGD for $\mathcal{O} \left( \frac{R^2\sigma^2}{\gamma^2\epsilon_1^2} + \frac{R^2L}{\gamma\epsilon_1} \right)$ iterations, then \ref{convex_complexity} ensures an output $X_1$ ( a random variable ) such that $\mathbb{E} \left[ f(X_1) - f(x^*) \right] \leq \epsilon_1$. Let $\mathcal{F}_1$ denote the $\sigma$-algebra generated by all the randomness in this stage.

Next,  we run SGD for another $\mathcal{O}\left( L \epsilon_1 \sigma^2\epsilon^{-4} \right)$ iterations, starting from $X_1$, then, according to a standard result in non-convex optimization \citet{ghadimi2013stochastic}, we can find a point (random variable) $X_2$ such that $\mathbb{E} \left[ \|\nabla f(X_2) \| \big| \mathcal{F}_1 \right] \leq \epsilon \left( \frac{f(X_1)-f(x^*)}{\epsilon_1} \right)^{\frac{1}{4}}$. This implies that 
\begin{equation}
    \mathbb{E} \|\nabla f(X_2) \| \leq \epsilon \mathbb{E} \left[  \left( \frac{f(X_1)-f(x^*)}{\epsilon_1} \right)^{\frac{1}{4}} \right] \leq \epsilon \left( \frac{\mathbb{E} \left[ f(X_1)-f^* \right]}{\epsilon_1} \right)^{\frac{1}{4}} \leq \epsilon
\end{equation}
The total iterations is then given by $\mathcal{O} \left( \dfrac{R^2\sigma^2}{\gamma^2\epsilon_1^2} + \dfrac{R^2L}{\gamma\epsilon_1} + L \epsilon_1 \sigma^2\epsilon^{-4} \right)$. Choose $\epsilon_1$ optimally and omitting lower order terms, we obtain the desired result.
\end{proof}

\section{Convergence Results for Smooth Strongly-Quasar-Convex Functions}
\label{sec_sqc}

In this section we turn to the study of smooth strongly-quasar-convex functions. As in Section \ref{sec_qc}, we consider two settings separately.

\subsection{Deterministic Setting}

In the deterministic setting, the following result was established in \citet{hinder2020near}.

\begin{theorem}
\label{hinder1}
( ~\cite[Theorem 1]{hinder2020near} )There exists an algorithm $\mathcal{A}_{sc}$ such that for $L$-smooth and $(\gamma ,\mu)$-strongly-quasar-convex functions, then the complexity of $\mathcal{A}_{sc}$ for finding an $\epsilon$-suboptimal point is $\mathcal{O}\left( \sqrt{\frac{\kappa}{\gamma^2}} \log \left( \frac{\kappa \Delta}{\epsilon} \right) \right)$, where $\kappa = L/\mu$.
\end{theorem}

An immediate consequence of the above result is the following:

\begin{corollary}
\label{simple}
The complexity of $\mathcal{A}_{sc}$ for finding $\epsilon$-stationary point is $\mathcal{O}\left( \sqrt{\frac{\kappa}{\gamma^2}} \log \left( \frac{\kappa L \Delta}{\epsilon^2} \right) \right)$. This also implies a complexity of $\mathcal{O}\left( \sqrt{\frac{\kappa}{\gamma^2}} \log \left( \frac{\sqrt{\kappa} LR }{\epsilon} \right) \right)$
\end{corollary}

\begin{proof}
Note that $L$-smoothness of $f$ implies that $\|\nabla f(x) \|^2 \leq L \left( f(x) - f(x^*) \right)$. The first statement follows from Theorem \ref{hinder1}. The second statement follows from the inequality $f(x_0)-f^* \leq \frac{L}{2}R^2$.
\end{proof}

\subsection{Stochastic Setting}
In the recent paper \citep{gower2020sgd}, the authors provide a convergence rate in terms of the distance between the output $x$ and the global minima $x^*$. Their analysis implies the following convergence rate of function value:

\begin{corollary}
\label{vanilla}
Suppose that we run vanilla SGD for $T$ with an appropriately chosen step size, then
\begin{equation}
    \mathbb{E}\left[ f(x_T)-f(x^*) \right] =  \mathcal{O} \left( \frac{L\sigma^2}{\mu^2\gamma^2T} \left( 1 + \log\left( \frac{\mu\gamma R \sqrt{T}}{\sigma} \right) \right) \right)
\end{equation}
\end{corollary}

\begin{proof}
According to ~\cite[Theorem D.2]{gower2020sgd} , we have
\begin{equation}
    \mathbb{E}\|x_T-x^*\|^2 \leq \left( 1-\alpha\mu\gamma \right)^T + \frac{2\alpha\sigma^2}{\mu\gamma} \leq \exp\left( -\alpha\mu\gamma T \right) R^2 + \frac{2\alpha\sigma^2}{\mu\gamma}
\end{equation}
Choosing $\alpha = \frac{1}{\mu\gamma T}\log\left( \frac{\mu^2\gamma^2R^2T}{2\sigma^2} \right)$ minimize the RHS of the above inequality. Thus we have
\begin{equation}
    \mathbb{E}\|x_T-x^*\|^2 \leq \frac{L\sigma^2}{\mu^2\gamma^2T} \left( 1 + \log\left( \frac{\mu\gamma R \sqrt{T}}{\sigma} \right) \right)
\end{equation}
Finally by $L$-smoothness we have
\begin{equation}
    f(x_T)-f(x^*) \leq \frac{L}{2}\|x_T-x^*\|^2
\notag
\end{equation}
and the conclusion follows.
\end{proof}

Next we refine the analysis and proves a new convergence guarantee which has better dependence on $\mu$ and $L$.
\begin{theorem}
\label{sqc_sto_complexity}
Suppose that we run vanilla SGD for $T > \max\left\{ \frac{3\sigma^2}{\gamma^2\mu^2R^2},\frac{6L}{\gamma^2\mu}\left( \log\left( \frac{2L\mu R^2}{\sigma^2} \right) + 1 \right) \right\}$ iterations with some fixed step size, then it can output a random point $X$ such that
\begin{equation}
    \mathbb{E}\left[ f(X)-f(x^*) \right] \leq \widetilde{\mathcal{O}}\left( \frac{\sigma^2}{\gamma^2\mu T} \right)
\end{equation}
\end{theorem}

\begin{proof}
Note that for any $x$ we have
\begin{equation}
    \begin{aligned}
    f(x^*) \leq f(x) - \frac{1}{2L}\|\nabla f(x)\|^2
    \end{aligned}
\notag
\end{equation}
Thus,
\begin{equation}
\label{recursion}
\begin{aligned}
    \mathbb{E}\|x_{t+1}-x^*\|^2 &\leq \mathbb{E}\left[ \|x_t-x^*\|^2 - 2\alpha (x_t-x^*)^T\nabla f(x) + \alpha^2 \left( \|\nabla f(x)\|^2+\sigma^2 \right)\right] \\
    &\leq \mathbb{E}\left[ (1-\gamma\mu\alpha)\|x_t-x^*\|^2 - 2\gamma\alpha (f(x_t)-f(x^*)) + \alpha^2 \left( \|\nabla f(x)\|^2+\sigma^2 \right)\right] \\
    &\leq \mathbb{E}\left[ (1-\gamma\mu\alpha)\|x_t-x^*\|^2 - 2\gamma\alpha (f(x_t)-f(x^*)) + 2\alpha^2L (f(x_t)-f(x^*)) + \alpha^2\sigma^2\right] \\
    &= \mathbb{E}\left[(1-\gamma\mu\alpha)\|x_t-x^*\|^2 - 2\alpha (\gamma-\alpha L)(f(x_t)-f(x^*)) + \alpha^2\sigma^2\right]
\end{aligned}
\end{equation}
Recursively apply the above inequality we have that
\begin{equation}
    2\alpha (\gamma-\alpha L) \sum_{t<T}(1-\gamma\mu\alpha)^{T-t-1}(f(x_t)-f(x^*)) \leq  \frac{\alpha\sigma^2}{\gamma\mu}+ (1-\gamma\mu\alpha)^TR^2 
\end{equation}
Suppose $X$ is a random variable such that $X=x_t, t=0,1,\cdots ,T-1$ with probability $(1-\gamma\mu\alpha)^{T-t-1}/Z$ where $Z$ is a normalizing constant, then we have
\begin{equation}
    \frac{2(\gamma-\alpha L)}{\gamma\mu}\left( 1-(1-\gamma\mu\alpha)^T \right) \mathbb{E}\left[ f(X)-f(x^*) \right] \leq \frac{\alpha\sigma^2}{\gamma\mu}+ (1-\gamma\mu\alpha)^TR^2 
\end{equation}
\begin{equation}
    \left( 1-(1-\gamma\mu\alpha)^T \right) \mathbb{E}\left[ f(X)-f(x^*) \right] \leq \frac{\alpha\sigma^2}{2(\gamma-\alpha L)} + \frac{\gamma\mu(1-\gamma\mu\alpha)^TR^2}{2(\gamma-\alpha L)}
\end{equation}
We choose $\alpha = \frac{1}{\gamma\mu T}\log\left( \frac{\gamma^2\mu^2TR^2}{\sigma^2} \right)$. When $T> \frac{3\sigma^2}{\gamma^2\mu^2R^2}$ we have $\alpha> \frac{1}{\gamma\mu T}$, so $1 - (1-\gamma\mu\alpha)^T > \frac{1}{2}$. On the other hand, if $T > \frac{6L}{\gamma^2\mu}\max\left\{ \log\left( \frac{2L\mu R^2}{\sigma^2} \right), 1 \right\}$, we have $\gamma - \alpha L > \frac{1}{2}\gamma$.
Thus for large $T$, $\mathbb{E}\left[ f(x)-f(x^*) \right] = \widetilde{\mathcal{O}}\left( \frac{\sigma^2}{\gamma^2\mu T} \right)$.
\end{proof}

\begin{theorem}
\label{sqc_sto_grad}
There exists an algorithm that achieves a complexity of $\mathcal{O}\left( \sqrt{\frac{L}{\mu}} \frac{\sigma^2}{\gamma\epsilon^2} + \frac{L}{\mu} \log\left( \frac{\gamma\mu \sqrt{L}R^2}{\epsilon^2} \right) \right)$ for finding $\epsilon$-stationary points for all small $\epsilon$.
\end{theorem}

\begin{proof}
We first run SGD for $\widetilde{\mathcal{O}}\left( \frac{\sigma^2}{\gamma^2\mu\epsilon_1}\right)$ to arrive at a (random) point $X_1$, then starting from $X_1$ we run SGD for another $\mathcal{O}\left( L\sigma^2\epsilon_1 \epsilon^{-4} \right)$ iterations. Then we can output a point with expected gradient norm smaller than $\epsilon$. The details are the same as the proof of Theorem \ref{qc_sto_grad}.
\end{proof}

\section{Extension to Non-smooth functions}
In this section we consider the optimization of possibly non-smooth quasar- and strongly quasar-convex functions. In this case we assume that (in expectation) the gradient is uniformly bounded, which is a standard assumption for non-smooth functions \citep{beck2017first,zhang2020complexity}. More precisely, in the deterministic setting we assume that $\|\nabla f(x)\| \leq G$ and in the stochastic setting we assume that the stochastic gradients $\nabla f(x,\xi)$ are unbiased and $\mathbb{E}\left[\|\nabla f(x,\xi)\|^2\right] \leq G^2$. These assumptions will be implicitly made throughout this section.

\begin{theorem}
Suppose that $f$ is $\gamma$-quasar-convex and we run SGD for $T$ iterations with step size $\alpha = \frac{R}{G\sqrt{T}}$. Then, we have
\begin{equation}
    \frac{1}{T} \sum_{t=0}^{T-1} \mathbb{E}\left[ f(x_t)-f(x^*) \right] \leq \frac{RG}{\gamma\sqrt{T}}
\end{equation}
\end{theorem}

\begin{theorem}
Suppose that $f$ is $(\gamma,\mu)$-quasar-strongly-convex and we run SGD for $T$ iterations with step size $\alpha_t = \frac{1}{\lambda t}$. Then we have

(1). $\mathbb{E}\left[ \|x_t-x^*\|^2 \right] \leq \frac{G^2}{\gamma^2\mu^2 t}$ 

(2). Suppose that $\widetilde{x}$ is uniformly chosen among $x_{T/2},\cdots ,x_T$, then we have $\mathbb{E}\left[ f(\widetilde{x})-f(x^*) \right] \leq \mathcal{O}\left( \frac{G^2}{\gamma\mu T} \right)$
\end{theorem}

The proof of these results are relatively standard and has little difference with those for convex functions, so we omit them here.

\section{Conclusion \& Future directions}
In this paper we study optimization of quasar- and strongly-quasar-convex functions with two different optimality criterions and in two different settings. However, there are still some interesting questions that remain unanswered. Firstly, it is unclear whether the dependency of our bounds on $\gamma$ is optimal. Indeed the discussion in Section \ref{comparison} suggests that they might be improved. Secondly we note that there exists another trick for finding $\epsilon$-stationary points for \textit{convex functions} in existing literature \citep{allen2018make,nesterov2012make}. The idea is to add a small perturbation to make the function strongly convex, which can be optimized very efficiently. This approach can yield complexities that match corresponding lower bounds. Unfortunately it cannot be applied to quasar-convex functions, since we cannot guarantee that $x^*$ is still the global minima after perturbation. It is thus unknown whether there exists other efficient approaches, or whether our approach is already optimal. Finally, we are looking forward to exploring convergence guarantees for other types of structured non-convex functions in the future.

\newpage
\bibliographystyle{plainnat}
\bibliography{reference}

\begin{thebibliography}{18}
\providecommand{\natexlab}[1]{#1}
\providecommand{\url}[1]{\texttt{#1}}
\expandafter\ifx\csname urlstyle\endcsname\relax
  \providecommand{\doi}[1]{doi: #1}\else
  \providecommand{\doi}{doi: \begingroup \urlstyle{rm}\Url}\fi

\bibitem[Allen-Zhu(2018)]{allen2018make}
Zeyuan Allen-Zhu.
\newblock How to make the gradients small stochastically: Even faster convex
  and nonconvex sgd.
\newblock In \emph{Advances in Neural Information Processing Systems}, pages
  1157--1167, 2018.

\bibitem[Arjevani et~al.(2019)Arjevani, Carmon, Duchi, Foster, Srebro, and
  Woodworth]{arjevani2019lower}
Yossi Arjevani, Yair Carmon, John~C Duchi, Dylan~J Foster, Nathan Srebro, and
  Blake Woodworth.
\newblock Lower bounds for non-convex stochastic optimization.
\newblock \emph{arXiv preprint arXiv:1912.02365}, 2019.

\bibitem[Beck(2017)]{beck2017first}
Amir Beck.
\newblock \emph{First-order methods in optimization}.
\newblock SIAM, 2017.

\bibitem[Bu and Mesbahi(2020)]{bu2020note}
Jingjing Bu and Mehran Mesbahi.
\newblock A note on nesterov's accelerated method in nonconvex optimization: a
  weak estimate sequence approach.
\newblock \emph{arXiv preprint arXiv:2006.08548}, 2020.

\bibitem[Carmon et~al.(2018)Carmon, Duchi, Hinder, and
  Sidford]{carmon2018accelerated}
Yair Carmon, John~C Duchi, Oliver Hinder, and Aaron Sidford.
\newblock Accelerated methods for nonconvex optimization.
\newblock \emph{SIAM Journal on Optimization}, 28\penalty0 (2):\penalty0
  1751--1772, 2018.

\bibitem[Cen et~al.(2020)Cen, Cheng, Chen, Wei, and Chi]{cen2020fast}
Shicong Cen, Chen Cheng, Yuxin Chen, Yuting Wei, and Yuejie Chi.
\newblock Fast global convergence of natural policy gradient methods with
  entropy regularization.
\newblock \emph{arXiv preprint arXiv:2007.06558}, 2020.

\bibitem[Ge et~al.(2017)Ge, Jin, and Zheng]{ge2017no}
Rong Ge, Chi Jin, and Yi~Zheng.
\newblock No spurious local minima in nonconvex low rank problems: a unified
  geometric analysis.
\newblock In \emph{Proceedings of the 34th International Conference on Machine
  Learning-Volume 70}, pages 1233--1242, 2017.

\bibitem[Ghadimi and Lan(2013)]{ghadimi2013stochastic}
Saeed Ghadimi and Guanghui Lan.
\newblock Stochastic first-and zeroth-order methods for nonconvex stochastic
  programming.
\newblock \emph{SIAM Journal on Optimization}, 23\penalty0 (4):\penalty0
  2341--2368, 2013.

\bibitem[Ghadimi and Lan(2016)]{ghadimi2016accelerated}
Saeed Ghadimi and Guanghui Lan.
\newblock Accelerated gradient methods for nonconvex nonlinear and stochastic
  programming.
\newblock \emph{Mathematical Programming}, 156\penalty0 (1-2):\penalty0 59--99,
  2016.

\bibitem[Goodfellow et~al.(2016)Goodfellow, Bengio, and
  Courville]{goodfellow2016deep}
Ian Goodfellow, Yoshua Bengio, and Aaron Courville.
\newblock \emph{Deep learning}.
\newblock MIT press, 2016.

\bibitem[Gower et~al.(2020)Gower, Sebbouh, and Loizou]{gower2020sgd}
Robert~M Gower, Othmane Sebbouh, and Nicolas Loizou.
\newblock Sgd for structured nonconvex functions: Learning rates, minibatching
  and interpolation.
\newblock \emph{arXiv preprint arXiv:2006.10311}, 2020.

\bibitem[Hinder et~al.(2020)Hinder, Sidford, and Sohoni]{hinder2020near}
Oliver Hinder, Aaron Sidford, and Nimit Sohoni.
\newblock Near-optimal methods for minimizing star-convex functions and beyond.
\newblock In \emph{Conference on Learning Theory}, pages 1894--1938. PMLR,
  2020.

\bibitem[Kingma and Ba(2014)]{kingma2014adam}
Diederik~P Kingma and Jimmy Ba.
\newblock Adam: A method for stochastic optimization.
\newblock \emph{arXiv preprint arXiv:1412.6980}, 2014.

\bibitem[Li and Yuan(2017)]{li2017convergence}
Yuanzhi Li and Yang Yuan.
\newblock Convergence analysis of two-layer neural networks with relu
  activation.
\newblock In \emph{Advances in neural information processing systems}, pages
  597--607, 2017.

\bibitem[Mei et~al.(2020)Mei, Xiao, Szepesvari, and Schuurmans]{mei2020global}
Jincheng Mei, Chenjun Xiao, Csaba Szepesvari, and Dale Schuurmans.
\newblock On the global convergence rates of softmax policy gradient methods.
\newblock \emph{arXiv preprint arXiv:2005.06392}, 2020.

\bibitem[Nesterov(2012)]{nesterov2012make}
Yurii Nesterov.
\newblock How to make the gradients small.
\newblock \emph{Optima. Mathematical Optimization Society Newsletter},
  \penalty0 (88):\penalty0 10--11, 2012.

\bibitem[Nesterov(2013)]{nesterov2013introductory}
Yurii Nesterov.
\newblock \emph{Introductory lectures on convex optimization: A basic course},
  volume~87.
\newblock Springer Science \& Business Media, 2013.

\bibitem[Zhang et~al.(2020)Zhang, Lin, Sra, and Jadbabaie]{zhang2020complexity}
Jingzhao Zhang, Hongzhou Lin, Suvrit Sra, and Ali Jadbabaie.
\newblock On complexity of finding stationary points of nonsmooth nonconvex
  functions.
\newblock \emph{arXiv preprint arXiv:2002.04130}, 2020.

\end{thebibliography}

\newpage
\clearpage
\appendix

\section{Formal descriptions of the setup}
\label{def}

In this section we introduce some useful concepts that allow us to rigorously describe the performance of a specific optimization algorithm.

\subsection{Optimization Oracle}

In order to describe the optimization process more conveniently, we assume that the optimization algorithm (only) has access to an \textit{oracle} which answers successive queries of the algorithm. 

In this paper we only consider two most commonly used oracle in the optimization literature.

\textbf{Deterministic Oracle} $\mathcal{D}(f)$. The algorithm sends a point $x$ to the oracle and the oracle responds with a pair $\left(f(x),\nabla f(x) \right)$.

\textbf{Stochastic Oracle} $\mathcal{S}(f,\sigma )$. The algorithm sends a point $x$ to the oracle and the oracle responds with random vector $g(x)$ such that $\mathbb{E}g(x) = f(x)$ and $\mathbb{E} \|g(x)-\nabla f(x)\|^2 \leq \sigma^2$.

For briefness we will refer to the optimization problem equipped with these two oracles as deterministic setting and stochastic setting, respectively. 

\subsection{Optimization Algorithms}

For a given oracle $\mathbb{O}$, we consider the set $\mathcal{A}\left( \mathbb{O} \right)$ consisting of all algorithms that works as follows: starting from a point $x_0$, it produces a (random) sequence $\{ x_t \}$ according the following recursive relation:
\begin{equation}
    x_t = \mathsf{A}_t \left( r, \mathsf{O}_0, \cdots , \mathsf{O}_{t-1} \right)
\end{equation}
where $\mathsf{O}_i$ is the oracle feedback at $x_i$, $r$ is a random seed and $\mathsf{A}_t$ is a deterministic mapping.

\subsection{Complexity Measures}

Consider a function class $\mathcal{F}$ and oracle class $\mathbb{O}$, let $\mathcal{P}\left[ \mathcal{F} \right]$ be the set of all distributions over $\mathcal{F}$ ,then for all $\epsilon > 0$ we define the complexity \textit{for finding approximate stationary point} as
\begin{equation}
\notag
    \sup_{\mathsf{O} \in \mathbb{O}} \sup_{P \in \mathcal{P}\left[ \mathcal{F} \right]} \inf_{\mathsf{A} \in \mathcal{A}(\mathsf{O})} \inf \left\{ T \in \mathbb{N} \big| \mathbb{E} \| \nabla f(x_T) \| \leq \epsilon \right\}
\end{equation}
and the complexity \textit{for finding approximate global minima} as
\begin{equation}
\notag
    \sup_{\mathsf{O} \in \mathbb{O}} \sup_{P \in \mathcal{P}\left[ \mathcal{F} \right]} \inf_{\mathsf{A} \in \mathcal{A}(\mathsf{O})} \inf \left\{ T \in \mathbb{N} \bigg| \mathbb{E} \left[ f(x_T)- \inf_{x \in \mathbb{R}^n} f(x) \right] \leq \epsilon \right\}
\end{equation}
where we omit the dependence of $x_T$ on $P$ and $\mathcal{A}$ in the above expressions.

\subsection{Notations}

Throughout this paper we let $x_0$ be the starting point of all the algorithms we consider, $x^*$ be the global minima of function $f$, and let $R, \Delta$ be upper bounds of the quantities $\|x_0-x^*\|$ and $f(x_0)-f(x^*)$ respectively. We use $\mathcal{O}$ to hide numerical constants and $\widetilde{\mathcal{O}}$ to hide $\log$ terms. 

The following two function classes appear regularly in the main paper:
\begin{equation}
\notag
    \mathcal{F}_{\mathtt{c}} \left( \gamma , R \right) = \left\{ f:\mathbb{R}^n \to \mathbb{R} : f \text{ is } L\text{-smooth and } \gamma \text{-quasar convex, and }\|x_0-x^*\| \leq R \right\}
\end{equation}
\begin{equation}
\notag
    \mathcal{F}_{\mathtt{sc}} \left( \gamma, \mu , R \right) = \left\{ f:\mathbb{R}^n \to \mathbb{R} : f \text{ is } L\text{-smooth and } (\gamma, \mu )\text{-quasar strongly-convex, and } \|x_0-x^*\| \leq R \right\}
\end{equation}

\section{Discussion of Theorem \ref{convex_subopt}}
\label{comparison}

In \citet{gower2020sgd}, the authors prove the following result:

\begin{theorem}
\label{gower}
( ~\cite[Theorem 4.1]{gower2020sgd} ) With appropriately chosen step sizes, SGD for finding $\epsilon$-optimal point of $L$-smooth, $\gamma$-quasar-convex functions has complexity $\mathcal{O}\left( \frac{R^2+c^2\sigma^2}{c\sqrt{T}} \right)$, where $c \in \left( 0,\frac{\gamma}{L} \right)$.
\end{theorem}

While their final convergence guarantee seems very different with ours, in their proof they actually show that, if we run SGD with step size $\alpha_1,\alpha_2,\cdots$, then we have
\begin{equation}
    \mathbb{E}\left[ f(\widetilde{x})-f(x^*) \right] \leq \frac{R^2}{2\sum_{t=0}^{T-1}\alpha_t \left( \gamma - L\alpha_t \right)} + \frac{\sigma^2\sum_{t=0}^{T-1}\alpha_t^2}{\sum_{t=0}^{T-1}\alpha_t \left( \gamma - L\alpha_t \right)}
\end{equation}
If we choose constant step size $\alpha$, the above becomes
\begin{equation}
\label{gower}
    \mathbb{E}\left[ f(\widetilde{x})-f(x^*) \right] \leq \frac{R^2}{2T\alpha (\gamma-L\alpha)} + \frac{\alpha\sigma^2}{\gamma -L\alpha}
\end{equation}
When $\alpha < \frac{\gamma}{2L}$ the above inequality is the same as \eqref{sum_Delta}, and we can proceed to select an optimal $\alpha$ to minimize \eqref{gower}.However, in \citet{gower2020sgd} the choice of $\alpha$ is not optimal.
\\

Another interesting feature of our proof of Theorem \ref{convex_subopt} is that it does not use the assumption that $x^*$ is the global minima. We hope that this feature can be used to design more efficient algorithms for making gradient small. Indeed, suppose that $f(x)$ is a $\gamma$-quasar convex, then for any point $x_0$, the function $g(x)=f(x)+\frac{\mu}{2}\|x_0-x\|^2$ satisfies
\begin{equation}
    g(x^*) \geq g(x)+\frac{1}{\gamma}\left\langle \nabla g(x), x^*-x \right\rangle + \frac{\mu}{2}\|x-x^*\|^2
\end{equation}
but $x^*$ may no longer be a global minima. Theorem \ref{convex_subopt} can be applied to establish convergence of $g$, however, we are unable to control the term $\|x_0-x\|^2$ during optimization process. A detailed study of possible solutions to this dilemma is left for future work.
\end{document}